\documentclass{lms}
\usepackage{amsbsy}
\usepackage{amsmath}
\usepackage{amssymb}
\newtheorem{theorem}{Theorem}[section] % 1st argument is your name for it
\newtheorem{lemma}[theorem]{Lemma}     % 2nd argument is what is printed
\newtheorem{corollary}[theorem]{Corollary}
\newtheorem{proposition}[theorem]{Proposition}
\newnumbered{assertion}{Assertion}    % 1st argument is your name for it
\newnumbered{conjecture}{Conjecture}  % 2nd argument is what is printed
\newnumbered{definition}{Definition}
\newnumbered{hypothesis}{Hypothesis}
\newnumbered{remark}{Remark}
\newnumbered{note}{Note}
\newnumbered{observation}{Observation}
\newnumbered{problem}{Problem}
\newnumbered{question}{Question}
\newnumbered{algorithm}{Algorithm}
\newnumbered{example}{Example}
\newunnumbered{notation}{Notation} % This is usually unnumbered
\title[Entire holomorphic curves on a Fermat surface]% end with percent
 {Entire holomorphic curves on a Fermat surface of low degree} % This is the full title of the paper
\author{Tuen-Wai Ng and Sai-Kee Yeung}

\dedication{Dedicated to Professor Walter K. Hayman on the occasion of his 90th birthday}

\classno{11B83 (primary), 30D05 (secondary)}

\extraline{The first author was partially supported by the RGC grant 17301115.  The second author was partially supported by a grant from the National Science Foundation}

\begin{document}
\maketitle

\begin{abstract}
The purpose of the paper is to study some problems raised by Hayman and Gundersen about the existence of non-trivial entire and meromorphic solutions for the Fermat type functional equation $f^n+g^n+h^n=1$. Hayman showed that no non-trivial meromorphic solutions and entire solutions exist when $n \ge 9$ and $n \ge 7$ respectively. By considering the entire holomorphic curves on the Fermat surface defined by $X^n+Y^n+Z^n=W^n$ on the complex projective space $\mathbb{P}^3$ and applying the method of jet differentials, we show that no non-trivial meromorphic solutions and entire solutions exist when $n \ge 8$ and $n \ge 6$ respectively.   In particular, this completes the investigation of non-trivial entire solutions for all
$n$ and respectively, meromorphic solutions for all cases except for $n=7$. Finally, for the generalized Fermat type functional equation $f^n+g^m+h^l=1$, we will also prove the non-existence of 
non-trivial meromorphic solutions when $1/n+1/m+1/l \le 3/8$, giving the strongest result obtained so far.   

\end{abstract}

%\part{Use this type of header for very long papers only}
% use lowercase except for proper names

\section{Introduction} % use lowercase except for proper names
\label{intro}

\noindent 

One of the most famous problems in number theory is the Fermat's Last Theorem which says that there is no natural numbers $x,y$ and $z$ such 
\begin{equation} \label{eqn:1} 
x^n+y^n=z^n
\end{equation}
for any natural number $n$ greater than $2$. The problem was eventually solved by Andrew Wiles, and the complete proof was published in 1995. 

The corresponding problem in one complex variable function theory is whether the equation (\ref{eqn:1}) has entire function solutions. This is equivalent to asking if the following functional equation has non-constant meromorphic solutions $f$ and $g$ on the complex plane $\mathbb{C}$:
\begin{equation} \label{eqn:2} 
f^n+g^n = 1
\end{equation}

It was proved by Iyer \cite{Iyer39} in 1939 (see also \cite{Gross66b}) that (\ref{eqn:2}) has no non-constant entire solutions when $n>2$  and when $n=2$, all entire solutions are of the form $f(z)=\cos(\alpha(z))$ and $g=\sin(\alpha(z))$, where $\alpha$ is a non-constant entire function. Gross \cite{Gross66a} showed in 1966 that (\ref{eqn:2}) has no non-constant meromorphic solutions when $n>3$  and when $n=2$, all the meromorphic solutions are of the form 
$$f(z)=\frac{2\beta(z)}{1+\beta(z)^2},\quad  g(z)=\frac{1-\beta(z)^2}{1+\beta(z)^2},$$
where $\beta$ is a meromorphic function.
For $n=3$, Baker \cite{Baker66} showed that all meromorphic solutions of (\ref{eqn:2}) are of the form $f(z)=F(\alpha(z))$ and $g=cG(\alpha(z))$ where $\alpha$ is an entire function, $F$ and $G$ are the elliptic functions $\frac{1+3^{-1/2}\wp'(z)}{2\wp(z)}$ and $\frac{1-3^{-1/2}\wp'(z)}{2\wp(z)}$ respectively. Here $c$ is a cubic root of unity and $\wp$ is the Weierstrass $\wp$ function.

It is then natural to ask what happens to the three term Fermat-type functional equation
\begin{equation} \label{eqn:3} 
f^n+g^n+h^n = 1
\end{equation}

Given any non-constant meromorphic function $f$, if we let $g=\omega_1f$ and $h=\omega_2$ where $\omega_1, \omega_2 \in \mathbb{C}$ such that $\omega_1^n=-1$ and $\omega_2^n=1$, then we get the trivial solution $(f,g,h)$ to (\ref{eqn:3}).
So by {\it non-trivial solutions} to (\ref{eqn:3}), we will mean solutions which are not of the form $(f(t),\omega_1f(t), \omega_2)$ or by permutation of the indices, where $\omega_1, \omega_2 \in \mathbb{C}$ such that $\omega_1^n=-1$ and $\omega_2^n=1$.\\

%we require that each of $f,g,h$ is a non-constant function on $\mathbb{C}$.\\

It was proved by Hayman \cite{Ha85} in 1985 that there is no non-trivial entire solutions to (\ref{eqn:3}) when $n\geqslant 7$ and there is no non-trivial meromorphic solutions to (\ref{eqn:3}) when $n\geqslant 9$.  For $n\leqslant 4$, Hayman \cite{Ha85} also showed  that there exist some meromorphic solutions for (\ref{eqn:3}) (actually Toda \cite{Toda71} also proved a more general results for the entire case in 1971 and Fujimoto proved the meromorphic case for meromorphic maps on $\mathbb{C}^k$ in \cite{Fujimoto74}). Hayman's proofs are based on Cartan's theory of holomorphic curves in projective spaces \cite{Cartan33}, which is a generalization of the value distribution theory of Nevanlinna. See \cite{GuHa04}, \cite{K98}and \cite{Lang87} for an introduction to Cartan's theory and \cite{AH14} for an attempt to sharpen Cartan's theory. In 2002, Ishizaki \cite{I02} gave a different proof of Hayman's results based on the classical Nevanlinna theory and he also pointed out that $f,g$ and $h$ must satisfy certain non-linear differential equation. \\

In 1998,  Gundersen \cite{Gu98} was able to construct meromorphic (elliptic) solutions for $n=6$ by expressing certain binary form as sum of powers of linear form (see also \cite{Toh11} for a detailed explanation of Gundersen's construction). Then in 2001, Gundersen \cite{Gu01} again constructed meromorphic solutions for the case $n=5$ using a result on the unique range sets of meromorphic functions.  
Examples of entire solutions also exist for $n \le 5$. They are given as follows where $\alpha$ is a non-constant entire function:\\

\noindent
{\bf Case n = 1}. $f, g$ non-constant entire,  $h = -f -g + 1$.\\
\noindent
{\bf Case n = 2}. $f = \dfrac {\alpha^2 - 2}{\sqrt{3}}, \,  g = \dfrac {(\alpha^2 + 1)i}{\sqrt{3}} , \, h = \sqrt{2} \alpha$\\
\noindent
{\bf Case n = 3}. Lehmer's example {\cite{Lehmer56}: $f = 9 \alpha^4, \,  g = - 9 \alpha^4 + 3 \alpha , \,  h = - 9 \alpha^3 + 1$\\
\noindent
{\bf Case n = 4}. Gross's example \cite{Gross66a}: 
$$
\begin{aligned}
f
&= 2^{1/4} (\sin^2 \alpha - \cos^2 \alpha  + i \sin \alpha  \cos \alpha), \\
g 
&= (- 1)^{1/4} (2 i \sin \alpha \cos \alpha + \sin^2 \alpha), \\
h 
& = (-1)^{1/4} (2 i \sin \alpha \cos \alpha - \cos^2 \alpha). 
\end{aligned}
$$

or Green's example \cite{Green75}: 

$$ f = 8^{-1/4} (e^{3 \alpha} + e^{-\alpha}), \,g = (-8)^{-1/4} (e^{3 \alpha} - e^{- \alpha}), \, h = (-1)^{1/4} e^{2\alpha}.$$

\noindent
{\bf Case n = 5}. Gundersen and Tohge's example \cite{GuToh04}: 
$$
\begin{aligned}
f
&= \frac{1}{3}[(2-\sqrt{6})e^{\alpha}+1+(2+\sqrt{6})e^{-\alpha}], \\
g 
&= \frac{1}{6}[\{(\sqrt{6}-2)+(3\sqrt{2}-2\sqrt{3})i\}e^{\alpha} + 2-\{(\sqrt{6}+2)-(3\sqrt{2}+2\sqrt{3}i\}e^{-\alpha}], \\
h 
& = \frac{1}{6}[\{(\sqrt{6}-2)+(2\sqrt{3}-3\sqrt{2})i\}e^{\alpha}+2-\{(\sqrt{6}+2)+(3\sqrt{2}+2\sqrt{3})i\}e^{-\alpha}]. 
\end{aligned}
$$

\noindent
Therefore, for the three term Fermat-type equation (\ref{eqn:3}), the remaining open problems are:\\

\noindent
{\bf Problem A:} Whether there exist non-trivial entire solutions of (\ref{eqn:3}) when $n=6$ ?\\

\noindent
{\bf Problem B:} Whether there exist non-trivial meromorphic (non-entire) solutions of (\ref{eqn:3}) when $n=7$ ?\\
 
\noindent 
{\bf Problem C:} Whether there exist non-trivial meromorphic (non-entire) solutions of (\ref{eqn:3}) when $n=8$ ?\\

The above three problems were asked by Hayman in many occasions. These problems are also mentioned in \cite{I02}, \cite{Gu03} and \cite{GuToh04}.
Very recently, Gundersen proposed to study these problems again in his problem list \cite{Gu16} (see Question 3.1 and 3.3 of this list). The main goal of this article is to settle Problem A and C by proving the following results.

\begin{theorem}
Suppose $n=6$.  Then there is no non-trivial entire solution to (\ref{eqn:3}).
\end{theorem}

\begin{theorem}
Suppose $n=8$.  Then there is no non-trivial meromorphic solution to (\ref{eqn:3}).
\end{theorem}

Recall that by a non-trivial solution to (\ref{eqn:3}), we mean a solution which is not of the form $(f(t),\omega_1f(t), \omega_2)$ or by permutation of the indices, where $\omega_1, \omega_2 \in \mathbb{C}$ such that $\omega_1^n=-1$ and $\omega_2^n=1$.\\

Hence combining the above theorems with known results in the literature, we conclude that
\begin{corollary}
(a).  There is no non-trivial meromorphic solution for (\ref{eqn:3}) in the case of $n\geqslant 8$, and there
are non-trivial transcendental meromorphic solution for (\ref{eqn:3}) in the case of $n\leqslant 6$.\\
(b).  There is no non-trivial entire solution for (\ref{eqn:3}) in the case of $n\geqslant 6$, and there
are non-trivial transcendental entire solution for (\ref{eqn:3}) in the case of $n\leqslant 5$.\\
\end{corollary}

Part a) of the above corollary partially answers a question of Fujimoto related to his Corollary 6.4 on meromorphic maps on $\mathbb{C}^k$ mentioned in page 273 of \cite{Fujimoto74}.\\

A complete proof of the non-existence of non-trivial meromorphic solutions in (a) and entire solutions in (b) independent of the results of \cite{Ha85} and \cite{I02} will be presented in Section 4 after we have introduced the general theory of jet differentials in Section 2 and some special holomorphic $2$-jet and log $2$-jet differentials in Section 3. To prove our main results, namely Theorem 1.1 and Theorem 1.2, we also need some estimates on the Nevanlinna characteristic functions for the pull back of the special holomorphic $2$-jet and log $2$-jet differentials constructed in Section 3. We provide these estimates in Section 5. We then prove Theorem 1.1 and Theorem 1.2 in Section 6. We also consider generalized Fermat functional equations $f^n+g^m+h^l=1$ in Section 7 and we will prove the non-existence of 
non-trivial meromorphic solutions of it whenever $1/n+1/m+1/l \le 3/8$ (Theorem 7.1). This gives the strongest result obtained so far. Finally, in Section 8, we mention a few related open problems that one may want to consider.  

\section{Holomorphic $2$-jet and log $2$-jet differentials}

To prove Theorem 1.1 and 1.2, we study the properties of entire holomorphic curves on the Fermat surface $S_n$
defined by 
\begin{equation}
X^n+Y^n+Z^n=W^n
\end{equation}
on the complex projective space $\mathbb{P}^3=\{[X:Y:Z:W]\}$.  On the affine part of $\mathbb{P}^3$ ($W \neq 0$), the equation is given by 

\begin{equation}\label{eqn:4} 
x^n+y^n+z^n=1
\end{equation}
where $x:=\frac{X}{W}$, $y:=\frac{Y}{W}$ and $z:=\frac{Z}{W}$.\\

We shall study some special holomorphic or meromorphic $2$-jet differentials on the Fermat surface $S_n$. So we recall the definition of 
$k$-th jet space and $k$-jet differentials (see \cite{Ru01}).

\begin{definition}
The $k$-th jet space $J_k(M) = \cup_{p\in M}J_k(M)_p$ is a bundle over an $n$ dimensional complex manifold $M$, where, for
each point $p\in M$, every element $v \in J_k(M)_p$ is a set of complex numbers
$(\xi_{j\alpha})_{1\le j \le k, 1\le \alpha \le n}$
with respect to a local coordinates $z_\alpha (1\le \alpha \le n)$ of $M$
in a neighborhood of $p$. Define $d^j z_\alpha: J_k(M) \to \mathbb{C}$ by $d^j z_\alpha (v) = \xi_{j\alpha}$ for
$v = (\xi_{j\alpha})_{1\le j \le k, 1\le \alpha \le n}$.
\end{definition}

\begin{definition}
A holomorphic $k$-jet differential $\omega$ (respectively
meromorphic $k$-jet differential) on an $n$ dimensional complex manifold $M$ assigns, at each
point $p\in M$, a function $\omega(p)$ on $J_k(M)_p$ such that, with local coordinates
$z_1, \cdots, z_n, \omega$ is locally a polynomial, with holomorphic 
(respectively meromorphic) functions as coefficients, in the variables 
$d^\ell  z_j  ( 1\le \ell \le k, 1\le j \le n)$ 
A meromorphic $k$-jet differential $\omega$ is said to be a log-pole $k$-jet differential if
it is locally a polynomial, with holomorphic functions as coefficients, in the
variables $d^\ell  z_j, d^\nu \log g_{\lambda} (1\le \ell \le k, 1 \le j \le n, 1\le \nu \le k, 1\le \lambda \le \Lambda)$,
where $g_{\lambda} ( 1\le \lambda \le \Lambda)$  are local holomorphic functions whose zero-divisors
are contained in a finite number of global nonnegative divisors of $M$.
\end{definition}

Thus, in terms of local coordinates $z_1, \ldots , z_n$, a meromorphic $k$-jet differential
is expressed in the form
$$\omega = \underset \nu \sum \omega_{\nu_{1, 1}\cdots \nu_{1, k}\cdots \nu_{n, 1}\cdots
\nu_{n, k}} (dz_1)^{\nu_{1,1}} \cdots  (d^k z_1)^{\nu_{1,k}} \cdots (dz_n)^{\nu_{n,1}} \cdots
(d^k z_n)^{\nu_{n,k}} $$
where the summation is over the $kn$-tuple
$$\nu = (\nu_{1, 1}\cdots \nu_{1, k}\cdots \nu_{n, 1}\cdots
\nu_{n, k})
$$
and $\omega_{\nu_{1, 1}\cdots \nu_{1, k}\cdots \nu_{n, 1}\cdots
\nu_{n, k}} $ is a meromorphic function locally defined. If $f:\mathbb{C} \to M$ is a holomorphic curve, then in terms of the local coordinates of $M$, $f$ naturally pulls back jet differentials and the above expression pulls back to

$$f^*\omega:=\underset \nu \sum \omega_{\nu_{1, 1}\cdots \nu_{1, k}\cdots \nu_{n, 1}\cdots
\nu_{n, k}}(f) (f_1')^{\nu_{1,1}} \cdots  (f_1^{(k)})^{\nu_{1,k}} \cdots (f_n')^{\nu_{n,1}} \cdots
(f_n^{(k)})^{\nu_{n,k}}$$

We will need the following results on the vanishing of pullback of jet differential. \\

Theorem A. (\cite{SY96}, \cite{SY97}) Let $M$ be a compact complex manifold of complex dimension $n$ and $D$ be an
ample divisor in $M$. Let $\omega$ be a $k$-jet differential on $M$ which vanishes on $D$ 
but is not identically zero on $M$. Then for any holomorphic map $f: \mathbb{C} \to M$, the
pullback $f^{*}\omega$ is identically zero on $\mathbb{C}$.\\

Theorem B. (\cite{SY97}) Let $k$ be a positive integer. Let $M$ be a compact complex manifold of complex
dimension $n$ and $D$ be an ample divisor in $M$. Let $Z_1,... ,Z_p$ be distinct irreducible
complex hypersurfaces in $M$. Let $\omega$ be a meromorphic $k$-jet differential on $M$ of
of at most log-pole singularity along $\cup_{i=1}^p Z_i$ such that $\omega$ vanishes on $D$
and is not identically zero on $M$. Then for any holomorphic map $f: \mathbb{C} \to M - \cup_{i=1}^p Z_i$,
the pullback $f^*\omega$ is identically zero on $\mathbb{C}$.\\

\section{Some general discussions}

One of the original goals of the project is to give a uniform treatment to all cases involved, namely, to prove non-existence of non-rational 
entire holomorphic curves (coming either from meromorphic functions or entire holomorphic functions on the affine part) for large degree,
and to explain the existence of such curves in low degree.  This is in principle possible with the use of holomorphic jet differentials.  
Denote by $S_n$ the Fermat surface of degree $n$ in $\mathbb{P}^3$.  Denote by $F_{k,m}$ the holomorphic jet bundle of order $k$ and 
homogeneous weight
$m$.  The following lemma is well-known, and can be found in \cite{GG80}.

\begin{lemma} (a). $S_n$ is a surface of general type if $n\geqslant 5$.\\
(b). Assume that $n\geqslant 5$.  Then $H^0(S_n, F_{1,m})=0$ for all $m>0$.\\
(c). Assume that $n\geqslant 5$. Sections of $H^0(S_n, F_{k,m})$ gives a birational mapping of $S_n$ if $k,m$ are sufficiently large.
\end{lemma}

\begin{proof}  (a) follows from the Adjunction Formula.  In fact, the canonical line bundle $K_{S_n}=(K_{\mathbb{P}^3}+nH)|_{S_n}=(n-4)|_{S_n}$ is ample
if $n>4$.

For (b), we observe that $H^0(S_n, F_{1,m})=H^0(S_n, S^m(\Omega_{S_n}))=0$ from a result of Sakai (cf. \cite{GG80}), here $S^m(\Omega_{S_n})$ denotes the space
of $m$-th symmetric differentials.

(c) is a result of Riemann-Roch Formula and is computed in \cite{GG80}, \S1.10-1.21.  Briefly, 
$$h^0(S_n, F_{k,m})-h^1(S_n, F_{k,m})+h^2(S_n, F_{k,m})=\chi((S_n, F_{k,m})$$
The right hand side is large, while $h^2(S_n, F_{k,m})$ vanishes from a vanishing theorem of Bogomolov, making use of the semi-stability of the tangent bundle of $S_n$.

\end{proof}

The following result is an immediate corollary.

\begin{proposition}
Suppose $n\geqslant 5$.  Then any entire holomorphic curve on $S_n$ lies in the integral curve of a certain ordinary differential equation on $S_n$.
\end{proposition}

\begin{proof}  Let $f:{\mathbb C}\rightarrow S_n$ be an entire holomorphic curve.
From part (c) of the previous lemma, we know that there exists a non-trivial holomorphic jet differential $\eta\in H^0(S_n, F_{k,m})$ vanishing on an ample divisor of $S_n$
if $k$ and $m$ are
sufficiently large.  From Theorem A, we conclude that $f^*\eta=0$.  This implies that the image of $f$ satisfies an ordinary differential equation of order $k$ on $S_n$. 

\end{proof}

\begin{remark}
The above theorem in principle equips with us a tool to locate all the entire holomorphic curves on $S_n$, by integrating out the differential equations involved.  The focus here
is not on complex hyperbolicity, that is, non-existence of entire holomorphic curves, but rather the properties of such curves.  The trouble is
that the differential equations involved are not explicit and hence difficult to work with.  Results of Section 4 shows how to construct some explicit jet differentials for $n\geqslant 8$ (vanishing 
on an ample divisor for $n\geqslant 9$).  In the case of $n\geqslant 9$, the integral curves are found and the image of entire holomorphic curves are determined.
\end{remark}

\section{Special jet differentials and entire holomorphic curves on Fermat surfaces of high degree }

We will apply Theorem A and B to some $2$-jet differentials obtained from (\ref{eqn:4}).

By taking derivatives of equation (\ref{eqn:4}), we obtain
\begin{eqnarray}
0&=&x^{n-1}dx+y^{n-1}dy+z^{n-1}dz\\
0&=&x^{n-1}D^2x+y^{n-1}D^2y+z^{n-1}D^2z
\end{eqnarray}
where $D^2F=d^2F+\frac{n-1}F(dF)^2$ 
for a function $F$.

Applying Crammer's rule to equations (2.2), (4.1) and (4.2), it follows that 
\begin{equation}\label{eqn:5} 
\frac{
\left|\begin{array}{cc}
d y&d z\\
D^2y&D^2z
\end{array}\right|}{x^{n-1}}
=\frac{
\left|\begin{array}{cc}
d z&d x\\
D^2z&D^2x
\end{array}\right|}{y^{n-1}}
=\frac{
\left|\begin{array}{cc}
d x&d y\\
D^2x&D^2y
\end{array}\right|}{z^{n-1}}
\end{equation}

Let $\Phi$ be the above expression.

We shall need the following properties of the $2$-jet differential $\Phi$.\\

\begin{lemma}
We have the following identity.
\begin{equation}\label{eqn:8} 
\Phi=\left|\begin{array}{ccc}
x&y&z\\
dx&d y&d z\\
D^2x&D^2y&D^2z
\end{array}\right|
=(xyz)M_{xyz},
\end{equation}
where 
$$M_{xyz}=\left|\begin{array}{ccc}
1&1&1\\
\frac{dx}x&\frac{dy}y&\frac{dz}{z}\\
\frac{D^2x}x&\frac{D^2y}y&\frac{D^2z}z
\end{array}\right|.
$$
Hence for
$$
M_{yz}=\left|\begin{array}{cc}
\frac{dy}y&\frac{dz}{z}\\
\frac{D^2y}y&\frac{D^2z}z
\end{array}\right|, \ 
M_{zx}=\left|\begin{array}{cc}
\frac{dz}{z}&\frac{dx}x\\
\frac{D^2z}z&\frac{D^2x}x
\end{array}\right|, \ 
M_{xy}=\left|\begin{array}{cc}
\frac{dx}x&\frac{dy}y\\
\frac{D^2x}x&\frac{D^2y}y
\end{array}\right|, $$
\begin{equation}\label{eqn:9} 
\Phi=\frac{(yz)M_{yz}}{x^{n-1}}=\frac{(zx)M_{zx}}{y^{n-1}}=\frac{(xy)M_{xy}}{z^{n-1}}=(xyz)M_{xyz}.
\end{equation}
\end{lemma}

\begin{proof} The first identity follows from (\ref{eqn:5}) and the fact that
\begin{eqnarray*}
\Phi&=&\frac{x^n\Phi+y^n\Phi+z^n\Phi}{x^n+y^n+z^n}\\
&=&\frac1{x^n+y^n+z^n}[
x\left|\begin{array}{cc}
d y&d z\\
D^2y&D^2z
\end{array}\right|
+
y\left|\begin{array}{cc}
d z&d x\\
D^2z&D^2x
\end{array}\right|
+
z\left|\begin{array}{cc}
d x&d y\\
D^2x&D^2y
\end{array}\right|]\\
&=&\left|\begin{array}{ccc}
x&y&z\\
dx&d y&d z\\
D^2x&D^2y&D^2z
\end{array}\right|,
\end{eqnarray*}
where we used the fact that $x^n+y^n+z^n=1$ from definition.  Hence (\ref{eqn:8}) follows and (\ref{eqn:9}) then follows from (\ref{eqn:5}) and (\ref{eqn:8}).

\end{proof}

\begin{lemma}
The  $2$-jet differential $xyz\Phi$ is holomorphic for $n\geqslant 8.$  Moreover, $xyz\Phi$ vanishes along an ample divisor 
on $S_n$ for $n\geqslant 9$.
\end{lemma}

\begin{proof}  Observe that $D^2x$ has only a simple pole at $x=0$, and similarly for $D^2y$ and $D^2z$ by 
permutation in $x, y, z$.  Hence by looking at the first term of (\ref{eqn:5}), we see that on the affine part, $xyz\Phi$ is holomorphic
except possibly at $x=0$.  On the other hand, by looking at the second (respectively third) term, $xyz\Phi$ is holomorphic
except at $y=0$ (respectively $z=0$).  Note that  $\{x=0, y=0, z=0\}$ has trivial intersection with $S_n$ on $\mathbb{P}^3$, since
$S_n$ is smooth.  Hence $xyz\Phi$ is holomorphic at the affine part of $S_n$.  

Consider now the pole order of $\Phi$ at $\infty$. For the first term on the right hand side of (\ref{eqn:5}), we may write
\begin{eqnarray}\label{eqn:6} 
\Phi&=&\frac1{x^{n-1}}(dyD^2z-dz D^2y)\nonumber \\
&=&\frac1{x^{n-1}}[(dyd^2z-dzd^2y)+(n-1)dydz(d\log z-d\log y).
\end{eqnarray}

 Suppose that the infinity is defined by $w=0$ in local coordinate.
It follows that we may consider the transformations 
\begin{equation}\label{eqn:7} 
x=\frac1w, \ \ y=\frac uw, \ \ z=\frac vw.
\end{equation}
It follows by direct computation that $dyd^2z-dzd^2y$ has a pole of order $3$ at $\infty$ and $dydz(d\log z-d\log y)$ has a pole of order $4$ at $\infty$.  Hence the numerator of $xyz\Phi$ has
pole order $7$ at $\infty$.  The denominator is $x^{n-1}$, giving rise to a zero of $\Phi$ of order $n-1$ at $\infty$.  Hence the pole
order of $xyz\Phi$ at $\infty$ is $8-n$.  Hence $xyz\Phi$ is holomorphic for $n\geqslant 8$.  Moreover, if $n\geqslant 9$, $xyz\Phi$
vanishes along the ample divisor given by the hyperplane at $\infty$.

\end{proof}

\begin{corollary} (\cite{Ha85})
For $n\ge 9$, there is no non-trivial meromorphic solution of (\ref{eqn:3}).
\end{corollary}

\begin{proof}  Consider the mapping $F=[f,g,h,1]: \mathbb{C}\rightarrow \mathbb{P}^3.$  By grouping the poles together, $F$ has a 
holomorphic representation $F=[a_1,a_2,a_3,a_4],$ where $a_i$ is an entire holomorphic function for each $i=1,2,3,4.$
Denote $(x_1,x_2,x_3)=(x,y,z)$ and $(f_1,f_2,f_3)=(f,g,h)$.   Hence $f_i=\frac{a_i}{a_4}$.
From definition, 
$F^*d^ix_j=\frac{d^if_j}{d\zeta^i}.$  In this way, $F^*\Phi$ is defined as well.\\

From  Lemma 4.2 and Theorem A, $F^*(xyz\Phi)=0$. Hence from (\ref{eqn:9}), unless $F(\mathbb{C})$ lies in a coordinate plane, we may assume that
$F^*M_{xy}=0$, since the former case can be handled easily.  This implies that $F(\mathbb{C})$ satisfies the differential
equation $M_{xy}=0$.  Hence
\begin{eqnarray}
&&dx(d^2y+\frac{n-1}y(dy)^2)-dy(d^2x+\frac{n-1}x(dx)^2)=0\nonumber  \\
&\Rightarrow&dxd^2y-dyd^2x=-(n-1)dxdyd\ln(\frac yx)\nonumber\\
&\Rightarrow&d(\frac{dy}{dx})=-(n-1)\frac{dy}{dx}\ln(\frac yx)\nonumber\\
&\Rightarrow&d\ln(\frac{dy}{dx})=-(n-1)\ln(\frac yx)\nonumber\\
&\Rightarrow&\frac{dy}{dx}=k_1(\frac yx)^{-(n-1)}\nonumber\\
&\Rightarrow&y^n=k_1x^n+k_2,\label{eqn:15}
\end{eqnarray}
where $k_1$ and $k_2$ are constants.  
Hence the image of $F$ is contained in the two equations (\ref{eqn:4}) and (\ref{eqn:15}).
One checks easily from the genus formula that unless $k_1=0$ or $k_2=0$, the genus of the curve cut out
by the two equations is at least $2$, which is hyperbolic. This will force $x$ and $y$ to be constant functions. Hence either $k_1=0$ or $k_2=0$.  In either case,
we conclude that the image of the curve lies in a rational curve of the form
$(f(t),\omega_1f(t), \omega_2,1)$ or by permutation of the indices, where $\omega_1, \omega_2 \in \mathbb{C}$ such that $\omega_1^n=-1$ and $\omega_2^n=1$. There only trivial solutions exist and we are done.

\end{proof} 

%The following lemma is a conclusion of Lemma 4.2. Note that the formal derivatives of a holomorphic jet differential is again a holomorphic jet differential of higher order.
 
% \begin{lemma}
%Let $n\geqslant 8$.   On $S_n$, $d^j(xyz\Phi)\in H^0(S_n, F_{2+j,3+j})$ for $j\geqslant0$.
% \end{lemma}

Consider  now the case of entire holomorphic solutions to equation (\ref{eqn:3}).  This is equivalent to existence of 
an entire holomorphic curve on the surface $S_n$ defined by $X^n+Y^n+Z^n=W^n$ avoiding the curve $W=0$.  To be consistent
with the discussions in the earlier sections, let us consider the equivalent problem of existence of entire holomorphic curve
on $S_n-C_Z$, where the curve $C_Z$ is defined by $\{Z=0\}$ on $S_n$, by switching the roles of $W$ and $Z$.  Hence we are looking for
the entire solutions $e,f,g$ to the equation
\begin{equation}\label{eqn:16} 
f^n+g^n+1=e^n
\end{equation}

\begin{lemma}
The $2$-jet differential $\frac{xy}z \Phi$ is a holomorphic log $2$-jet divisor with logarithmic poles along the divisor $C_Z$ of $M$ for $n\geqslant 6.$  Moreover, $\frac{xy}z \Phi$ vanishes along an ample divisor on $S_n$ for $n\geqslant 7$.
\end{lemma}

\begin{proof} From (\ref{eqn:5}) and (\ref{eqn:6}), we see that $\frac{xy}z \Phi$ is holomorphic at $x=0$ and $y=0$, but has a log pole along $z=0$.  Now from
the last paragraph in the proof of Lemma 4.2, we see that the pole of $\frac{xy}z \Phi$ has order $6-n$ at $\infty$ corresponding
to $W=0$.  Hence we conclude that on $S_n$, the  $2$-jet differential $\frac{xy}z \Phi$ is a holomorphic everywhere except a log pole
along $D_Z$ if $n\geqslant 6$.  Moreover, it vanishes along the ample divisor given by $W=0$ when $n \geqslant 7$.
\end{proof}

We immediately have the following corollary.

\begin{corollary} (\cite{Ha85})
There is no non-trivial entire solution $e(\zeta), f(\zeta), g(\zeta)$ to the equation (\ref{eqn:16}) for $n\geqslant 7$.
\end{corollary}

\begin{proof} This follows from Lemma 4.4 and Theorem B (the Schwarz Lemma for log-jet sections as stated in Theorem 3a in \cite{SY97}) and the arguments used in the proof of Corollary 4.3.
\end{proof} 

\begin{remark}
Switching the roles of $Z$ and $W$ in the above arguments, we may still study on the affine part solutions to
$$f^n+g^n+h^n=1$$
for $n\geqslant 6$ and consider the two jet of the form 
$xyz\Phi$.  In such case, the argument of the above shows that  $xyz\Phi$ has log poles at $\infty$ so that  the arguments of the above
still forces $F^*\Phi_{xyz}=0$ for $n\geqslant 7$.  Under the transformation $(x,y,z,1)\rightarrow (x,y,1,w)$ in two different standard affine coordinates for ${\mathbb P}^3$,
the jet differential $xyz\Phi_{xyz}$ corresponds $\frac{xy}w\Phi_{xyw}$, the one given in Lemma 4.4, where $\Phi_{xyz}$ is $\Phi$ discussed
earlier in affine coordinates $(x,y,z,1)$ and $\Phi_{xyw}$ is the similar expression in affine coordinate $(x,y,1,w)$.

\end{remark}

\section{On $T(r,F^*(xyz\Phi))$}

%Throughout this section, we are assuming that $n=8$.

Let us recall some standard notations from value distribution theory (see for example, \cite{Ru01} and \cite{Siu95}).  In the following, let $g$ be a function on $\mathbb{C}$,
and $\eta$ be a one form on $\mathbb{C}$.  We define
\begin{eqnarray*}
\mathcal{A}_r(g)&=&\frac1{2\pi}\int_0^{2\pi}g(re^{i\theta})d\theta\\
\mathcal{I}_r(\eta)&=&\int_0^r\frac{d\rho}\rho\int_{|z|<\rho}\eta.
\end{eqnarray*}

Let $F:\mathbb{C}\rightarrow \mathbb{P}^n$ be a holomorphic mapping to $\mathbb{P}^n$ and $\omega$
be the K\"ahler form of the Fubini-Study metric on $\mathbb{P}^n$.  Let $D$ be a hypersurface on $\mathbb{P}^n$.  Let $\infty$
denote the hypersurface at $\infty.$  We define

\begin{eqnarray*}
T(r,F)&=&\mathcal{I}_r(F^*\omega)\\
N(r,F,D)&=&\mathcal{I}_r(F^*D),
\end{eqnarray*}
where the latter is interpreted as a current.

A meromorphic function $f$ can be considered as a mapping $f:\mathbb{C} \rightarrow \mathbb{P}^1$.  In such case,
$N(r,f,\infty)$ or $N(r,f,0)$ are defined as above when $\infty$ and $0$ are regarded as divisors on $\mathbb{P}^1$.\\

Again we consider the mapping $F=[f,g,h,1]: \mathbb{C}\rightarrow \mathbb{P}^3$ as in the proof of Corollary 4.3 and adopt the notation there.

 \begin{proposition} (a). Let $n=6$.  Assume that $f, g, h$ are entire and $F^*M_{xyz}\neq0$.  Then
$T(r,F^*(xyz\Phi))=\mathcal{O}(\log(T(r,F)))$ for $r$ outside a set $E$ of finite measure with respect to $\frac{dr}{r}$. \\
(b).  Let $n=8$ and assume that $f, g, h$ are meromorphic and $F^*M_{xyz}\neq0$.  Then the same conclusion holds.
\end{proposition}

\begin{proof}  (a). As in Remark 2 of Section 4, $\Omega:=xyz\Phi$ is a special holomorphic log-2-jet differential on $S_n$ with log-divisor at $D$ which is the divisor
at $\infty$.  Let us cover $S_n$ by a finite number of
open sets $U_\alpha$, $\alpha=1,\dots,N$.  We choose local holomorphic 
coordinates $(x,w)$ on $U_\alpha$ so that $D$ is defined by $w=0$ if $D\cap U_\alpha\neq\emptyset$.   $\Omega$ is then a polynomial expression in 
$$dx,d\log w, d^2xd\log w-dx d^2\log w$$
with coefficients which are meromorphic functions.     Hence $\Omega=\sum_{i,j,k}a_{ijk}dx^id\log w^j(d^2xd\log w-dx d^2\log w)^k$ on $U_\alpha$,
where $i+j+3k=\ell$ is a constant.
Note that the expressions make sense even if $D\cap U_\alpha=\emptyset$.
 The proximity term ${\mathcal A}_r(\log^+|F^*\Omega|)$ refers to the proximity integral of the sum of the pull back of all the coefficients.
 As $F$ is entire,
we still have $N(r,F^*\Omega,\infty)=0$.  %Consider now the proximity term ${\mathcal A}_r(\log^+\Omega)$.
Denote by $\omega$ the K\"ahler form on $S_n$ induced from the Fubini-Study metric on ${\mathbb P}^3$.  By considering partition of
unity $\{\rho_\alpha\}$ subordinated to the collection of open sets $\{U_\alpha\}$ and writing $\Omega=\sum_{\alpha=1}^N\rho_\alpha\Omega$, we have
\begin{eqnarray*}
{\mathcal A}_r(\log^+|F^*\Omega|)&\leqslant &{\mathcal A}_r(\log^+(N\max_{1\leqslant \alpha\leqslant N} (\rho_\alpha |F^*\Omega|)))\\
&\leqslant &\sum_{\alpha=1}^N {\mathcal A}_r(\log^+(\rho_{\alpha} |F^*\Omega|)+\log N))\\
\end{eqnarray*}

%we may without loss
%of generality simply estimate each ${\mathcal A}_r(\log^+(\rho_\alpha\Omega))$ and take the sum.  For simplicity of presentation, we would just neglect $\rho_\alpha$
%in the following estimates.
\begin{eqnarray*}
{\mathcal A}_r(\log^+(\rho_{\alpha} |F^*\Omega|))&=&{\mathcal A}_r(\log^+|F^*(\sum_{i,j,k}a_{ijk}(dx)^i(d\log w)^j(d^2xd\log w-dx d^2\log w)^k|))\\
&\leqslant&{\mathcal A}_r(\log^+(1+c_1\sum_{j+3k\leqslant \ell}(|F|^2+|F^*(d\log w)^j|+|F^*(d^2xd\log w-dx d^2\log w)^k)|)\\
&\leqslant&\log(c_2+c_3\mathcal A}_r(|F|^2)+c_4{\sum_{j\leqslant 2, jm\leqslant\ell}{\mathcal A}_r(|F^*(d^j\log w)^m|)\\
&\leqslant&\log(c_2+c_4r^\epsilon({\mathcal I}_r(F^*\omega))^{1+\delta})\Vert_{\epsilon,\delta}\\
&=&{\mathcal O}(\log T(r,F))\Vert_{\epsilon,\delta},
\end{eqnarray*}
where we used concavity of $\log$ in the third line, the Logarithmic Derivative Lemma and Calculus Lemma (cf. \cite{Siu95}, Lemma 1.1.3) in the fourth line.
For example,
\begin{eqnarray*}{\mathcal A}_r(\log^+|F^*(dw)|^2)&\leqslant& {\mathcal A}_r(\log^+|F^*(d\log w)|^2)+ {\mathcal A}_r(\log^+|F^*(w)|^2)\\
&\leqslant& {\mathcal A}_r(\log^+|F^*(d\log w)|^2)+ {\mathcal A}_r(\log^+|F|^2)
\end{eqnarray*}
and so forth.
Part (a) of the lemma follows from the First Main Theorem and the fact that the counting function  $N(r,F^*\Omega,\infty)=0$.

The proof of (b) is exactly the same, replacing log jet differentials by the usual holomorphic jet differentials.
\end{proof}

%We remark that the proposition above gave an alternative, perhaps more conceptual proof of the Proposition 5.1 and 6.1 in Ishizaki's paper \cite{I02}.

\section{Proof of Theorem 1.1 (n=6) and Theorem 1.2 (n=8) }

\begin{proof} We first prove Theorem 1.2 ($n=8$). Suppose there are meromorphic functions $f,g,h$ satisfying
$$f^8+g^8+h^8=1.$$
If $F^*M_{xyz}=0$, then we will have (4.8) for $n=8$ and we are done.
Now if $F^*M_{xyz}\neq0$, then we know from (4.5) of Lemma 4.1 and Proposition 5.1(b) that 
$$p=F^*(\frac{x^2y^2}{z^6}M_{xy})=F^*(\frac{xy}{z^6}(dxD^2y-dyD^2x))$$
is a small function, i.e. $T(r,p)= \mathcal{O}(\log(T(r,F)))$. The condition $F^*M_{xyz}\neq0$ also implies that all $f, g$ and $h$ are non-constant.\\

\noindent
From p.427 of \cite{Siu95}, we know that 
$$\frac13(T(r,f)+T(r,g)+T(r,h)) + \mathcal{O}(1) \le T(r,F) \le T(r,f)+T(r,g)+T(r,h) + \mathcal{O}(1).$$
Hence, $T(r,p)=o(T(r,f)+T(r,g)+T(r,h))$.
From now on, we will denote by $S(r)$ any quantity that satisfies $S(r)=o(T(r,f)+T(r,g)+T(r,h))$, $r \to \infty, r \notin E$, here $E$ is a set of finite measure with respect to $\frac{dr}{r}$. \\

From p. 82 of \cite{I02}, we know for $n=8$ that 
$$T(r,f)+S(r)=T(r,g)+S(r)=T(r,h)+S(r).$$

To prove Theorem 1.2, we need the following 

\begin{lemma}
Assume that $n=8$ and $(f,g,h)$ is a meromorphic solution of (\ref{eqn:3}). 
Let $p=F^*(xyz\Phi)$ and $Z(p)$ be the zero set of $p$.  Then the following statements hold.\\
(a). None of any two of $f,g,h$ have common zero on $\mathbb{C}-Z(p)$.\\
(b). $f,g,h$ can only have simple zeros on $\mathbb{C}-Z(p)$.\\
(c). $f, g, h$ share common poles on $\mathbb{C}-Z(p)$, and the pole order has to be $1$.\\
(d).  Let $f_i$ be one of $f, g, h$.  Then $m(r,f_i^{(k)})=S(r)$ for all $k \ge 0$ and $i$
and $m(r,\frac1{f_i})=S(r)$
\end{lemma}

\begin{proof}
We first notice that following the argument used to get (4.8), we have
\begin{equation}\label{eqn:i1} 
(\ln(\frac{g'}{f'}))'=-7(\ln(\frac gf))'+p\frac{h^6}{fgf'g'}.
\end{equation}

Hence, \begin{equation}\label{eqn:i2} 
\frac{(g^8)'}{(f^8)'}=Ae^{\int p\frac{h^6}{f'g'fg}d\zeta}
\end{equation}
for some constant $A$.

For (a), suppose $f, g$ have a common zero at $0$ apart from $Z(p)$, the zero set of $p$. Then $h(0)\neq0$ from equation (\ref{eqn:3}).
It follows that $p(\zeta)\frac{h^{6}}{f'g'fg}$ has a pole of order at least two at $\zeta=0$.  The right hand side of (\ref{eqn:i2}) then has an essential singularity, contradicting the
fact that the left hand side is a meromorphic function.

For (b), if $f(0)=0$, it follows that $h(0)\neq0$ and $g(0)\neq0$.  Hence if $f$ has a zero of order at least $2$ at a point $\zeta\in \mathbb{C}-Z(p)$, the right hand side of
(\ref{eqn:i2}) would give rise to an essential singularity.  Again we reach a contradiction.

For (c), let $\zeta=0$ be a pole of one of $f, g, h$.  Without loss of generality, we may assume that the pole order of $h$ at $\zeta=0$ is the largest among $f, g, h$.
Then unless the order of $h$ at $0$ is $1$ and both $f$ and $g$ have pole order precisely $1$ at $0$, the expression $p(\zeta)\frac{h^{6}}{f'g'fg}$ would have a pole
of order greater than $1$ at $0$.  Again, the right hand side of (\ref{eqn:i2}) has an essential singularity and we reach a contradiction.

For (d), we notice from p.82 of \cite{I02} that $m(f_i), m(\frac1{f_i})=S(r)$.  Hence $m(r,f')\leqslant m(r,\frac{f'}f)+m(r,f)=S(r)$ from Logarithmic Derivative
Lemma.  From induction, it follows that $m(r,f^{(k)})=S(r)$.
\end{proof}

Now from (\ref{eqn:i1}), we have
$$(\ln(\frac{g'}{f'}))'+7(\ln(\frac gf))'=p\frac{h^6}{fgf'g'}$$
and hence 
$$\frac1p(fgf'g''-f''fgg''+7(g')^2ff'-7(f')^2gg')=h^6.$$ 

Clearly, a pole $z_0$ of $h^6$ will either be a pole of $fgf'g''-f''fgg''+7(g')^2ff'-7(f')^2gg'$ or a zero of $p$. If $z_0$ is outside $Z(p)$, then $z_0$ is a pole of $fgf'g''-f''fgg''+7(g')^2ff'-7(f')^2gg'$ and hence a pole of at least one of the four terms $fgf'g'',f''fgg'',(g')^2ff',(f')^2gg'$. In fact, $z_0$ should be a pole for each of $fgf'g'',f''fgg'',(g')^2ff',(f')^2gg'$ as $f,g,h$ share a pole of order one when the pole is outside $Z(p)$ by Lemma 6.1(c). Hence we have $6N(r,h)=N(r,h^6) \le N(r,\frac1p)+N(r,fgf'g'')\le S(r) + N(r,f)+N(r,g)+N(r,f')+N(r,g'')$ as by the First Fundamental Theorem, we have $N(r,\frac1p) \le T(r,p)+ \mathcal{O}(1) = S(r)$. By Lemma 6.1(d), we then have 
$$6T(r,h) + S(r)\le T(r,f)+T(r,g)+T(r,f')+T(r,g'')+S(r).$$

From $T(r,f)+S(r)=T(r,g)+S(r)=T(r,h)+S(r)$ and the well-known fact that $T(r,w') \le (1+\varepsilon) T(r,w)$, it follows that
$$6T(r,h) +S(r) \le (4 + \varepsilon) T(r,h) + S(r),\, r \to \infty, r \notin E.$$
This is a contradiction and we conclude the proof of Theorem 1.2.
\end{proof}

\begin{proof} We now prove Theorem 1.1 ($n=6$) and let us consider entire solution $f,g$ and $h$ for $$f^6+g^6+h^6=1.$$
If $F^*M_{xyz}=0$, then we will have (4.8) for $n=6$ and we are done.
Now suppose $F^*M_{xyz} \neq 0$ which implies that all $f, g$ and $h$ are non-constant.

Note that $xyz\Phi_{xy}$ is a log jet differential with log-pole at the hyperplane at $\infty$, which is the same as Lemma 4.4 after a change of coordinates.  Proposition 5.1(a) in this case implies that 
$p:=xyz\Phi_{x,y}$ is a small function, i.e. $T(r,p)= \mathcal{O}(\log(T(r,F)))=S(r)$. 
 
Now following the argument used to get (4.8), we have
$$(\ln(\frac{g'}{f'}))'=-5(\ln(\frac gf))'+p\frac{h^6}{fgf'g'}.$$

Note that, from p.84 of \cite{I02}, for $n=6$, we again have  $$T(r,f)+S(r)=T(r,g)+S(r)=T(r,h)+S(r),$$
where $r \to \infty, r \notin E$. 

Let $a = \Big(\ln \big(\frac{g'}{f'}\big) \Big)'$ and $b = 5 \Big(\ln \big(\frac{g}{f}\big) \Big)'$. 
Hence, by the Logarithmic Derivative Lemma and the fact that $T(r,w') \le (1+\varepsilon) T(r,w)$, we have
\begin{align*}
m(r,a) & = o\Big( T \big(r, \frac{g'}{f'}\big) \Big)
 = o\big( T (r, g') + T(r,f')) + O(1)\big)\\
& = o\big( (1+\varepsilon) T(r, g) +
(1+\varepsilon') T(r,f) + O(1)\big)\\
& = S(r), r \to \infty, r \notin E;
\end{align*}

\begin{align*}
m(r,b) & = o\Big( T \big(r, \frac{g}{f}\big) \Big)
 = o\big( T (r, g) + T(r,f)) + O(1)\big)\\
& = S(r), r \to \infty, r \notin E.
\end{align*}

Since $h^6=\frac{a+b}p \, fgf'g'$ and $T(r,p)=S(r)$, we have
\begin{align*}
6m(r,h) +S(r) & = m(r,h^6) + S(r) \\
& \le m \Big(r, \frac{a+b}p \Big) 
+ m(r,f) + m(r,f') + m(r,g) + m(r,g') + S(r)\\
& = (4 + \varepsilon) T(r,h) + S(r),\, r \to \infty, r \notin E.
\end{align*}

Since $h$ is entire, we have $T(r,h)=m(r,h)$ and hence 
$$6T(r,h) +S(r) \le (4 + \varepsilon) T(r,h) + S(r),\, r \to \infty, r \notin E.$$
This is a contradiction and we conclude the proof of Theorem 1.1.

\end{proof}

%We remark that the proposition above gave an alternative, perhaps more conceptual proof of the Proposition 5.1 and 6.1 in Ishizaki's paper \cite{I02}.

\begin{remark}
Ishizaki \cite{I02} showed that meromorphic solutions for the $n=8$ case would satisfy differential equations of the form $W(f^8,g^8,h^8)=a(z)(f(z)g(z)h(z))^6$ where $W(f_1,f_2,f_3)$ is the Wronskian determinant of $f_1,f_2,f_3$. He also showed that entire solutions for the $n=6$ case satisfy differential equations of the form $W(f^6,g^6,h^6)=b(z)(f(z)g(z)h(z))^4$.
One can check that $(\ln(\frac{g'}{f'}))'+7(\ln(\frac gf))'=p\frac{h^6}{fgf'g'}$ is the same as $W(f^8,g^8,h^8)=a(z)(f(z)g(z)h(z))^6$ where $p=\frac{a}{64}$. Similarly, $(\ln(\frac{g'}{f'}))'+5(\ln(\frac gf))'=p\frac{h^6}{fgf'g'}$ is the same as $W(f^6,g^6,h^6)=b(z)(f(z)g(z)h(z))^4$, where $p=\frac{b}{36}$. From Proposition 5.1 and 6.1 of Ishizaki's paper \cite{I02}, we know that $a=S(r)$ and $b=S(r)$. Hence $p=S(r)$ when $n=8$ and $n=6$.
\end{remark}

\section{Generalized Fermat functional equations}

In this section, we will explain briefly how the method of jet differentials can be applied to study the meromorphic solutions of the generalized Fermat functional equations:
\begin{equation}
f^n+g^m+h^l=1, \label{eqn:g}
\end{equation} where $n\geqslant m\geqslant l$.\\

Note that Hayman's Theorem is equivalent to saying that for $n=m=l$, there is no non-trivial meromorphic solution of (\ref{eqn:g}) if
$\frac1n +\frac1m + \frac1l < \frac38$. In 2003, Hu, Li and Yang \cite{HuLY03} showed that there is no non-trivial meromorphic solution of (\ref{eqn:g}) if
$\frac1n +\frac1m + \frac1l < \frac13$. The condition is improved to $\frac1n +\frac1m + \frac1l \le  \frac13$ in \cite{YZ10} and then to 
$\frac1n +\frac1m + \frac1l  <  \frac{25}{72}=\frac19 +\frac19 + \frac18$ by Yi and Yang in 2011. In this section we shall prove the following strongest result so far.

\begin{theorem}
There is no non-trivial meromorphic solution of $f^n+g^m+h^l=1$ if
$$\frac1n +\frac1m + \frac1l  \le \frac 38$$
\end{theorem}

By modifying Green's example \cite{Green75} mentioned in Section 1, we have $f^4+g^4+h^n=1$ where 
$f = 8^{-1/4} (e^{3 z} + e^{-z}), \,g = (-8)^{-1/4} (e^{3 z} - e^{- z}), \, h = (-1)^{1/n} e^{\frac{8}{n}z}$. So the constant $\frac38$ cannot be replaced by a number greater than $\frac 12$. Note that the surface associated with $f^4+g^4+h^n=1$ is not smooth. As we will see, one needs to pay special attention to the type of singularities a surface can have if one wants to apply the method of jet differentials. \\ 

To prove Theorem 7.1, it would be more convenient to consider 
\begin{equation} 
f^n+g^m+h^l+1=0, \label{eqn:g1}
\end{equation} where $n\geqslant m\geqslant l$.\\

We then have the following 

\begin{proposition}
Assume that a generalized Fermat surface $S$ defined by equation (\ref{eqn:g1}) is a normal surface with isolated singularities.  
Assume that $n\geqslant 9$.  Then any entire holomorphic curve on $S$ lies on an algebraic curve.
\end{proposition}

\begin{proof} (a).  Let us assume first that the surface $S$ is smooth.
By taking derivatives of equation (\ref{eqn:g1}), we obtain
\begin{eqnarray}
0&=&nx^{n-1}dx+my^{m-1}dy+lz^{l-1}dz\\
0&=&nx^{n-1}D^2x+my^{m-1}D^2y+lz^{l-1}D^2z
\end{eqnarray}
where 
$$D^2x=d^2x+\frac{n-1}x(dx)^2, \ D^2y=d^2y+\frac{m-1}y(dy)^2 , \ D^2z=d^2z+\frac{l-1}z(dz)^2 .$$

 It follows that 
\begin{equation}\label{eqn:g2} 
\frac{
\left|\begin{array}{cc}
d y&d z\\
D^2y&D^2z
\end{array}\right|}{nx^{n-1}}
=\frac{
\left|\begin{array}{cc}
d z&d x\\
D^2z&D^2x
\end{array}\right|}{my^{m-1}}
=\frac{
\left|\begin{array}{cc}
d x&d y\\
D^2x&D^2y
\end{array}\right|}{lz^{l-1}}
\end{equation}

Let $\Phi$ be the above expression.  From construction,
$xyz\Phi$ is holomorphic except perhaps at $\infty$.
Studying the first expression above after multiplying by $xyz$, we note that the numerator has pole order $7$ in $\infty$
as in the proof of Lemma 4.2.  It follows that if $n>8$, the expression $xyz\Phi$ would have no pole at a generic point of the divisor
at $\infty$ and in fact, vanish along this ample divisor.   Hence it has singularity at most along a codimension $2$ subset on $S$.
From Riemann Extension Theorem, we may extend $(xyz)\Phi$ as holomorphic jet differentials on $S$.\\

Let $F=[x,y,z,1]:\mathbb{C}\rightarrow S$ be an entire holomorphic curve.
Apply Theorem A, we have $F^*((xyz)\Phi)=0$. This implies that if $F(\mathbb{C})$ does not lie in a coordinate plane, we will have the differential
equation $F^*(\Phi)=0$.  Hence
\begin{eqnarray}
&&dx(d^2y+\frac{m-1}y(dy)^2)-dy(d^2x+\frac{n-1}x(dx)^2)=0\label{eqn:gb}\\
&\Rightarrow&d(\frac{dy^m}{dx^n})(dx^n)^2=0\nonumber\\
&\Rightarrow&d(y^m)=c_1{d(x^n)}\nonumber\\
&\Rightarrow&y^m=c_1{x^n}+c_2\nonumber
\end{eqnarray}
and  $F(\mathbb{C})$ lies on an algebraic curve.

\medskip
(b). Consider now the general case as stated in the statement of Proposition 7.2.    Let $\{R_i, i=1,\dots,r\}$ be the set of
normal singularities of $S$.  The construction above gives rise to a holomorphic two-jet differential on $S-\cup_{i=1}^r\{R_i\}$.
Let $D$ be the hyperplane $w=0$ of $\mathbb{P}^3$.  Suppose that $R_i$ lies in the affine part $S\cap \mathbb{C}^3=S\cap(\mathbb{P}^3-D)$. Then $xyz\Phi$ is holomorphic on
$S\cap(\mathbb{P}^3-D)-\cup_{i=1}^r\{R_i\}$.  Written in terms of local coordinates as in equation (\ref{eqn:gb}), the coefficients of $xyz\Phi$ is bounded in a neighborhood of $R_i$
and extends across $R_i$ since $S$ is normal.  Hence $xyz\Phi$ extends across $R_i$.   The same argument applies to $S\cap D$ in terms of a chart
around $D$ corresponding to a different standard affine charts on $\mathbb{P}^3$.  We conclude that $xyz\Phi$ extends across all the singular points.  By considering
a resolution of singularities $p:\widehat{S}\rightarrow S$ of $S$ and pulling back $xyz\Phi$, we obtain a holomorphic $2$-jet differential $p^*(xyz\Phi)$
vanishing along an ample divisor.   The rest of argument as in (a) allows us to conclude the proof.

\end{proof}
\medskip
\begin{corollary}
Entire holomorphic curve on the following surfaces are projective algebraic.\\

\noindent
a) $X^n+Y^n+Z^{n-1}W=W^{n}$ for $n\geqslant 9$\\
b) $X^n+Y^{n-1}W+Z^{n-1}W=W^n$ for $n\geqslant 9$.
\end{corollary}

\medskip
\begin{proof}
According to case 76 of the appendix of \cite{Heijne16}, the surface given by (a) is smooth.
Moreover, from case 48 of the appendix of \cite{Heijne16}, the surface given by (b) is smooth except for ADE singularities $[0:\omega_n:1:0]$, where $\omega_n$ is a $(n-1)$-th root of $-1$. It is well-known that a surface with only ADE singularities is normal following Serre's criterion, cf. Appendix to \S3 of \cite{Reid87}, we can then apply Proposition 7.2 to conclude the proof.
\end{proof}
\bigskip

In general we may  consider a Delsarte surface as discussed in \cite{Shioda86} and \cite{Heijne16} because the generalized Fermat surface is a special type of Delsarte surfaces. 

\begin {proposition}  Let $S$ be a Delsarte surface with isolated ADE singularities and sufficiently large degree.  Then any entire holomorphic curve on
$S$ is projective algebraic.  In particular, this applies to all $83$ classes of surfaces tabulated in \cite{Heijne16} with $n\geqslant 9$.
\end{proposition}

\begin{proof}  Again, as surfaces with only ADE singularities are normal, the result follows from Proposition 7.2.  The condition $n\geqslant 9$
makes sure that the construction as given in equation (\ref{eqn:gb}) leads to existence of a holomorphic $2$-jet differential vanishing along an ample divisor. 
%({\bf Need to double check for all the cases})
\end{proof}

%\begin{proof} 
To prove Theorem 7.1, it remains to show that the condition $\frac1n +\frac1m + \frac1l  <  \frac{25}{72}=\frac19 +\frac19 + \frac18$ can be replaced by the condition $\frac1n +\frac1m + \frac1k  \le \frac83$.  Therefore, we only need to consider the cases $(n,m,l)=(9,9,8),(9,8,8)$ or $(8,8,8)$. The case $(8,8,8)$ has already been covered by Theorem 1.2. We learn from the proof of Corollary 7.3 (a) that the surface associated with the case $(9,9,8)$ is smooth and we can follow the proof of Proposition 7.1 to conclude that $g^m=c_1{f^n}+c_2$ which will lead to a contradiction as shown in the proof of Corollary 4.3. Finally, from the proof of Corollary 7.3 (b), the case $(9,8,8)$ will correspond to a surface with ADE singularities only and the surface will therefore be normal and we can again follow the proof of Proposition 7.1 to conclude that $g^m=c_1{f^n}+c_2$ which will again lead to a contradiction.

%\end{proof}
\section{Some related open problems}

From the proof of Corollary 4.5 and Theorem 1.1 for the entire solutions of the Fermat functional equation (2.2), we know that when $n \ge 6$, one can only have solutions of the form $(f(t),\omega_1f(t), \omega_2)$ or by permutation of the indices, where $\omega_1, \omega_2 \in \mathbb{C}$ such that $\omega_1^n=-1$ and $\omega_2^n=1$. We have the same conclusion for meromorphic solutions when $n \ge 8$. While there are non-trivial entire solutions when $n\le 5$, we only have the existence of non-trivial meromorphic solutions when $n \le 6$. It seems to us that the case $n=8$ for meromorphic solutions and $n=6$ for entire solutions are really the boundary cases and therefore we conjecture that there are non-trivial meromorphic solutions for $n=7$. \\   

For generalized Fermat functional equations, it would be interesting to know whether the number $\frac 38$ in the condition $\frac1n +\frac1m + \frac1l  \le \frac 38$ can be replaced by something bigger. 
In this paper, we only consider meromorphic functions for the generalized Fermat functional equations. It is natural to consider entire solutions as well. Note that Toda \cite{Toda71} has proved a general result which implies that if $f,g,h$ are non-constant entire functions satisfying $f^n+g^m+h^l=1$, then $\frac1n +\frac1m + \frac1l  \ge \frac 12$. The explicit entire solution for $f^4+g^4+h^n=1$ for any natural number $n$ mentioned in Section 7 shows that Toda's result is optimal.\\

In page 421 of \cite{Shioda86}, one can find, for a given Delsarte surface, a construction of a Fermat surface which is a covering of the given Delsarte surface. For example, for the Delsarte surface $X^n + Y^{n-1}W + z^{n-1}W + W^n =0$, the Fermat surface of degree $n(n-1)$, $S_{n(n-1)}$,  is its covering and the covering map is $\varphi = (x^{n-1}, \dfrac{y^n} w , \dfrac{z^n} w , w^{n-1})$.
In fact, one can check on $S_{n(n-1)}$,
$$\varphi^n_1 + \varphi_4\varphi^{n-1}_2 + \varphi_4\varphi_3^{n-1} + \varphi_4^n = x^{n(n-1)} + w^{n-1} \cdot \dfrac{y^{(n-1)^n}}{w^{n-1}} + w^{n-1} \cdot \dfrac{z^{(n-1)^n}}{w^{n-1}} 
   + w^{(n-1)n}=0. $$

So the image of $\varphi$ is lying on the Delsarte surface $X^n + y^{n-1}W + z^{n-1}W + W^n =0$.
For $n=3$, we have $S_{6}$ and by the covering map $\varphi$, if we have a non-trivial solution for $f^6+g^6+h^6=1$, then we can get a non-trivial solution for $x^3+y^2+z^2=1$ (which is known to have meromorphic or entire solutions). Gundersen's example for $f^6+g^6+h^6=1$ in \cite{Gu98} will then give a new solution for $x^3+y^2+z^2=1$. Therefore, it is still interesting to look for new non-trivial meromorphic solutions for the Fermat functional equations when $n=6$.\\

 Of course one may also consider the four term Fermat-type equations. In fact, Hayman \cite{Ha85} has already considered the existence of solution for the following general Fermat-type functional equation in connection with his study of Waring's problem in function theory:
\begin{equation} \label{eqn:h1} 
f_1^n+f_2^n+...+f_k^n=1
\end{equation}

Note that in Hayman's work for the Waring's problem in function theory \cite{Ha85},  one only needs to consider the Waring-type functional equation:
 
\begin{equation} \label{eqn:h2} 
f_1^n(z)+f_2^n(z)+...+f_k^n(z)=z
\end{equation}

It would be also interesting to see how the method of jet differentials can be applied to study these more general equations. One can refer to \cite{Ha95}, \cite{Ha14} and \cite{GuHa04} for some known results of (\ref{eqn:h1}) and (\ref{eqn:h2}).

%\begin{acknowledgements}\label{ackref}
%\end{acknowledgements}

\medskip

\affiliationone
{Tuen-Wai Ng\\
Mathematics Department\\
The University of Hong Kong\\
Pokfulam Road\\
Hong Kong, China}
\email{ntw@maths.hku.hk}\\

\affiliationone
{Sai-Kee Yeung\\ 
Mathematics Department\\
Purdue University\\
West Lafayette, IN  47907\\
USA}
\email{yeung@math.purdue.edu}
   
\end{document}